\documentclass[12pt,reqno]{amsart}

\usepackage{amssymb,amsmath,amsthm,mathtools,wasysym,calc,verbatim,enumitem,tikz,pgfplots,hyperref,url,mathrsfs,fullpage,bbm,comment}
\usepackage[noadjust]{cite}
\usepackage{comment}
\mathtoolsset{showonlyrefs}
\pgfplotsset{compat=1.18}
\newtheorem{theorem}{Theorem}

\newtheorem{lemma}[theorem]{Lemma}
\newtheorem{claim}[theorem]{Claim}
\newtheorem{corollary}[theorem]{Corollary}

\newtheorem{fact}[theorem]{Fact}

\newtheorem{obs}[theorem]{Observation}

\newtheorem*{claim*}{Claim}

\theoremstyle{remark}
\newtheorem*{remark*}{Remark}

\usepackage{thmtools, thm-restate}

\numberwithin{theorem}{section}

\renewcommand{\phi}{\varphi}

\renewcommand{\leq}{\le}
\renewcommand{\geq}{\ge}

\newcommand{\cE}{\mathcal E}

\def\1{\mathbbm{1}}

\usepackage{tikz-cd}
\usetikzlibrary{positioning,arrows.meta,calc}

\renewcommand{\le}{\leqslant}
\renewcommand{\ge}{\geqslant}

\newcommand{\EE}{\mathbb E}

\newcommand{\PP}{\mathbb P}

\DeclareMathOperator{\Pois}{Pois}
\DeclareMathOperator{\Mult}{Mult}

\newcommand{\Bin}{\operatorname{Bin}}

\newcommand{\NN}{\mathbb{N}}

\title{The Multicolour Size Ramsey Number of a Path}
\author{Csongor Beke, Anqi Li and Julian Sahasrabudhe}
\email{\{cb2138, jdrs2\}@cam.ac.uk, aqli@stanford.edu}
\thanks{Julian Sahasrabudhe is supported by European Research
Council (ERC) Starting Grant “High Dimensional Probability and Combinatorics”, grant
No. 101165900}

\begin{document}
\begin{abstract}
 In this paper, we determine the $r$-colour size Ramsey number of the path $P_k$, up to constants. In particular, for every $r \geq 2$ and $k \geq 200\log r$, we have
    \[ \widehat{R}_r(P_k)=\Theta((r^2 \log r) \, k).\] 
Perhaps surprisingly, we do this by improving the lower bound on $\widehat{R}_r(P_k)$.
\end{abstract}
\maketitle

\vspace{-2em}

\section{Introduction}

We write $G \to_r H$ if every $r$-colouring of $G$ contains a monochromatic copy of $H$. The $r$-\emph{colour size Ramsey number} of a graph $H$, denoted $\widehat{R}_r(H)$, is the smallest number of edges in a graph $G$ for which $G \to_r H$. That is,
\[ \widehat{R}_r(H)= \min \{ e(G) : G \to_r H \}.\] 
This is a natural relative of the classical $r$-colour Ramsey number of a graph $H$, denoted $R_r(H)$, which is the minimum number of \emph{vertices} in a graph $G$ for which $G\rightarrow_r H$.

The size Ramsey number was originally introduced by Erd\H{o}s, Faudree, Rousseau and Schelp~\cite{EFRS78} in the 1970s and has since become a central topic in  Ramsey theory (see \cite{CFS15}).  One of the most celebrated results in this line is the 1983 work of Beck~\cite{B83} who showed, perhaps surprisingly, that the size Ramsey number of a path is as small as it could possibly be: linear in its number of vertices. That is, he showed that
\[ \widehat{R}_2(P_k)= \Theta(k),\] thereby settling a question of Erd\H{o}s.

While Beck's work also implies that $\widehat{R}_r(P_k)= \Theta_r(k)$, it leaves open the interesting question of the dependence of $\widehat{R}_r(P_k)$ on the number of colours $r$, a problem which has gained some traction in the community in recent years. To date, the best-known bounds are 
\begin{equation}\label{eq:upper}
    \hspace{.8em} c_1 r^2k \leq \widehat{R}_r(P_k) \leq c_2r^2(\log r) k,
\end{equation}
where $c_1,c_2 >0$, $r$ is fixed and $k$ is sufficiently large. The upper bound is due to Krivelevich~\cite{K19} via a clever application of the depth first search algorithm to the host graph $G(C_1rk, C_2\log r/k)$ for some suitable constants $C_1, C_2 > 1$, and the implicit constants has since been improved by Dudek and Pra\l at \cite{DP17}. The lower bound is due to Dudek and Pra\l at \cite{DP17}, though the implicit constant has been improved by Krivelevich \cite{K19}. In this paper, we determine $\widehat{R}_r(P_k)$ up to constant factors. 

\begin{theorem}\label{thm:main}
    For any $r$ and $k \geq 200\log r$, we have 
    $$\widehat{R}_r(P_k)=\Theta\left((r^2 \log r) \,k\right).$$  
\end{theorem} 

Our approach in this paper is based on a novel randomised edge-colouring strategy that avoids generating large (in a suitable sense) monochromatic connected components. In particular, we define a random process that removes, from a given graph $G$, a $P_k$-free graph in each round. One interesting feature of this random process is that it ``improves'' as the process runs. 

\section{Proof overview}\label{sec:sketch}
In this section, we begin by sketching the main idea in the proof and then go on to describe the setup and some high-level intuition behind our random colouring process. 

\subsection{A first attempt}\label{sec:sketch-main} 

Given a graph $G$ with 
\begin{equation}\label{eq:c0}
    e(G) \leq c_0 (r^2\log r)\, k,
\end{equation}
for some $c_0$ to be determined and $r,k$ sufficiently large, we iteratively find large $P_k$-free subgraphs of $G$, assign an unused colour for them and remove them from $G$. 

We now motivate a first attempt where we aim to partition $G = H_1 \cup H_2 \cup \cdots \cup H_r$ so that for each colour class $i \in [r]$, there is a further partition \[V(H_i) = X_1^{(i)}\cup X_2^{(i)}\cup \ldots \cup X_q^{(i)},\] where
\[q = \frac{3v(G)}{k}, \qquad |X_j^{(i)}|\leq \frac{k}{2}, \hspace{0.9em} \text{and} \hspace{0.9em}  H_i =G[X_1^{(i)}]\cup \dots \cup G[X_q^{(i)}],\]
for each $i\in [r]$ and $j \in [q]$. That is, the only edges in the $i$th colour class $H_i$ are between vertices from the same component $X_j^{(i)}$. This is clearly a monochromatic $P_k$-free colouring, as each connected monochromatic component has size at most $k/2$. Here we construct the $i$th colour class $H_i$ iteratively by considering \[G=F_1 \supset F_2 \supset \dots \supset F_r \supset F_{r+1}=\emptyset,\] where $F_{i+1} =  F_i \setminus H_i$ for $i \in [r]$. Here $V(F_i) = V(G)$ for $i \in [r]$, with the union of the $H_i$ covering the whole graph. 

A naive first attempt might be to uniformly at random partition the vertices of $V(G)$ into $q$ parts. Each edge would then have probability around $1/q$ of being contained in a given $H_i$, so $e(H_i) \approx e(F_i)/q$. In particular, we would have 
\[e(F_{r+1})=e(G)\left(1- \frac{1}{q} \right)^r\simeq e(G)\exp\left(- \frac{r}{q} \right)\simeq e(G)e^{-1},\] 
so even after $r$ colours, we still have a constant fraction of the edges left uncoloured (which is not enough for this strategy). 

The crucial observation behind this paper is that, in the context of this problem, we can in fact do significantly better. That is, we will be able to remove $P_k$-free subgraphs $H_i$ with $e(H_i) \gg e(F_i)/q$. We now turn to outline a randomised greedy approach that finds one such colour class $H_i$ with many edges.

\subsection{The random greedy approach}\label{sec:random-greedy}

We make two key changes to the above strategy. Earlier, we described forming the colour classes $H_j$ by a union of disjoint connected components, each of at most $k-1$ vertices, to ensure each $H_j$ is $P_k$-free. The first key change we make is that we aim for a slightly weaker constraint: we ensure that each connected component is an induced bipartite graph $G[A_i,B_i]=\{uv\in E(G):u\in A_i, v\in B_i\}$ and use the following observation to guarantee that each $G[A_i,B_i]$ is $P_k$-free: 
\begin{equation}\label{eq:pk-free}
    |A_i| < k/2 - 1 \;  \Longrightarrow \; G[A_i, B_i] \not \supset P_k,
\end{equation}
since any path alternates between the two parts.

Another benefit of this bipartite structure is that it gives us a natural way to improve upon the naive randomisation strategy in the previous section. We choose the sets $\{A_i \}_i$ as before, uniformly at random. The other key change is that instead of choosing $\{ B_i\}_i$ as a uniformly random partition of the remaining vertices, we assign the remaining vertices to the $\{ B_i \}_i$ in a greedy fashion to capture more edges in $\bigcup_i G[A_i, B_i]$. This ensures that repeating this randomised greedy colouring $r$ times gives a genuine $r$-edge colouring of $G$.

 In more detail, we first extract a suitable bipartite subgraph with vertex set $A\cup B$ such that $e(A,B)$ is large relative to $e(G)$. For now, we can think about this as $e(A,B) \geq e(G)/2$ and $|A|=|B|=|V(G)|/2$. For some value of $q$ to be determined later, we find partitions
\begin{equation}\label{eq:partition}
    A = A_1 \cup \cdots \cup A_q\qquad \text{and}\qquad B =B_1\cup \cdots \cup B_q,
\end{equation}
with the property that for each $i\in [q]$ the bipartite induced subgraph $G[A_i,B_i]$ of $G$ between $A_i$ and $B_i$ is $P_k$-free. We then define
\begin{equation}\label{eq:H-def}
    H= G[A_1, B_1] \cup \cdots \cup G[A_q,B_q]
\end{equation}
to be our desired $P_k$-free subgraph.

The randomised greedy algorithm used to identify the partition $\{A_i\}_i \cup \{B_i\}_i$ is as follows. We first partition $A = A_1 \cup \cdots \cup A_q$ by assigning each vertex of $A$ to one of the $A_i$ uniformly at random. Then we define for each $i \in [q]$,
\begin{equation}\label{eq:process}
    B_i = \{ x \in B : |N(x) \cap A_i| \text{ is maximum among all } i \in [q]\},
\end{equation}
breaking ties arbitrarily. The key here is that, due to the heavy tails of the Poisson distribution, $H$ will induce many more edges than if we were to have chosen both of the $A_i$ and $B_i$ uniformly at random.

To analyse the number of edges in $H$, we are led to the well-studied ``balls-and-bins'' problem \cite{G81, RS98,MU17}. Focusing on a vertex $v\in B$ with $\deg_{G[A,B]}(v) = d$, we obtain the following problem: given $q$ bins (the sets $A_1,\dots,A_q$), and $d$ balls ($N(v) \cap A$), we throw each ball into one of the $q$ bins uniformly and independently at random. How many balls are there in the bin that has the most balls? That is, given a random function $f\colon [d] \to [q]$, we define
\begin{equation}\label{eq:MW-def}
    M_{q,d} = \max_{i \in [q]} |f^{-1}(i)| \qquad \text{and} \qquad W(q,d)=\mathbb{E} \hspace{0.1em} M_{q,d}/d.
\end{equation}

Thus to understand the number of edges in $H$, we need lower bounds on $W(q,d)$, or equivalently $\mathbb{E} M_{q,d}$. For this, we note that the number of balls in a given bin has distribution $\mathrm{Bin}(d,1/q)$, which can be well approximated by $\mathrm{Pois}(d/q)$. To estimate $M_{q,d}$, we consider the threshold $t$ for which the expected number of bins with at least $t$ balls is of order $1$. That is, we choose $t$ so that
\begin{equation}\label{eq:Pois-approx}
q \cdot \mathbb{P}(\mathrm{Pois}(d/q) \ge t) \approx 1.
\end{equation}
Thus, we expect that with constant probability there exists a bin with at least $t$ balls, and in particular $\mathbb{E} M_{q,d} \gtrsim t$. The main reason why the random greedy approach improves on the naive approach is the heavy tail of the Poisson distribution. Solving the relation \eqref{eq:Pois-approx} in the relevant range leads to the condition
$t(\log (tq/d))\gtrsim \log q$. Therefore, if we define
\begin{equation}\label{eq:W_0def}
    W_0(q, d)=\frac{x}{10q} \qquad \text{where} \qquad x\log x= \frac{q \log q}{d},
\end{equation}
we expect the following lemma to hold, the proof of which is deferred to Appendix~\ref{sec:appMbound}.

\begin{restatable}{lemma}{ballsandbins}\label{lem:ballsunified}
    For $d,q \in \NN$ with $q >1$, let $W_0(q,d)$ be defined as in \eqref{eq:W_0def}. Then \[W(q,d)\geq W_0(q,d).\]
\end{restatable}
Note that $W_0(q, d)=x/(10q)$, so $x$ can be interpreted as the ``factor of improvement'' over the naive approach. We also note that Lemma~\ref{lem:ballsunified} is essentially a sharp bound and one can really think of $W_0$ and $W$ as essentially interchangeable.

 Before we bound the number of edges in $H$, we note a basic fact about the function $W_0(q,d)$. The simple proof is deferred to Appendix~\ref{sec:appWprop}.

\begin{restatable}{fact}{mono}\label{fact:W_0mono}
    For $q,d\in \NN$, and $W_0(q,d)$ defined as in \eqref{eq:W_0def}, we have that $W_0(q,d)$ is monotonically decreasing in both $q$ and $d$. 
\end{restatable}
We can now bound the expected number of edges in $H$. 
\begin{obs}\label{obs:H}
    $\mathbb{E} \, e(H) \geq  e(A,B) \cdot W_0\left( q, \Delta(G)\right)$.
\end{obs}
\begin{proof}
    A single vertex $v \in B$ of degree $d(v) = \deg_{G[A,B]}(v)$ contributes on expectation $\mathbb{E}\hspace{0.1em} M_{q, \hspace{0.1em}d(v)}=W(q,d)\cdot d(v)$ edges to $H$ (since its neighbours in $A$ are assigned independently to the bins $A_i$). It follows that
\[ \mathbb{E} \hspace{0.2em}e(H) = \sum_{v\in B} d(v)\cdot W\left(q,\, d(v) \right)\geq\sum_{v\in B} d(v)\cdot W_0\left(q,\, d(v) \right)\geq e(A,B) \cdot W_0\left( q, \Delta(G)\right),\]
where we used Lemma~\ref{lem:ballsunified} and the monotonicity in $d$ of $W_0(q,d)$ by Fact~\ref{fact:W_0mono}. 
\end{proof}

In our application, we choose $q = 3|A|/k$, so that on expectation for each $i \in [q]$, we have 
\begin{equation}\label{eq:Ai}
  \mathbb{E}\,|A_i| = |A|/q < k/2 - 1,  
\end{equation}
as required by \eqref{eq:pk-free}.

Combining \eqref{eq:Ai} and Observation~\eqref{obs:H}, it is then not too difficult (see Lemma~\ref{lem:bipartite-finding} for further details) to show with positive probability that $H$ as constructed in \eqref{eq:H-def} gives the $P_k$-free graph $H$ with 
\begin{equation}\label{eq:withconcentration}
    e(H) \geq \frac{2e(A,B)}{3} \cdot W_0(q, \Delta(G)).
\end{equation}

As the definition of $W_0(q,d)$ is a bit hard to work with, it will be more convenient to apply the following Fact.

\begin{restatable}{fact}{lambdaballsandbins}\label{fact:lambdaballsunified}
    There exists $0<\lambda_0<1$ so that the following holds. For any $r\geq 2$, $\lambda\in (0,\lambda_0)$ so that $d=\lambda r\log r$ is a positive integer, and  $q\leq 20r/\lambda$ we have
    \[W_0(q,d)\geq \frac{120}{\lambda^{0.9}r}.\]
\end{restatable}

It is worth noting that the factor of $\lambda^{0.9}$ is crucial in the above lemma, as it is this factor that allows us to do better than the naive random approach of Section~\ref{sec:sketch-main}. There, we were only able to guarantee $\Omega(1/r)$ fraction of edges in a $P_k$-free graph, while we can now, in fact extract $\Omega(1/(r\lambda^{0.9}))$ fraction of edges in each round. We next study a special case of the problem to give an illustration of our bounds in action. 

\subsection{Toy example}

In this toy example, we assume that our graph  $G$ is such that
\[ v(G) = 100rk, \quad\text{and}\quad e(G) \leq \lambda_0 r^2 \log r\, k, \]
 with $\lambda_0$ is as in Fact~\ref{fact:lambdaballsunified} and we exhibit an $r$ colouring of $G$ without monochromatic $P_k$. Note that this falls in line with the upper bound constructions of $G(C_1rk, C_2 \log r/k)$ in \eqref{eq:upper}.
 
We would like to apply the above strategy of extracting $P_k$-free subgraphs in turn, each forming a new colour class. Recall that the first step of the algorithm is to pass to a suitable bipartite subgraph. In this case, we take $V(G)=A\cup B$, so that $|A|= 50rk$ and $e(A,B)\geq e(G)/2$ (for instance, by considering a uniformly random partition). For illustration purposes, we make the additional assumption that $G[A,B]$, and in fact all of these bipartite subgraphs that we encounter, are regular. This is, of course, a strong assumption and much of the technical work in this paper is about how to get around this.

In the previous subsection, we outlined a randomised greedy algorithm to find a $P_k$-free subgraph $H \subset G$ for which $e(H)\geq e(G)/3\cdot W_0(q,\Delta(G))$ as in \eqref{eq:withconcentration}. We define $G_0=G$, and having defined some $G_i$, we let $H_i$ be the $P_k$-free subgraph found by applying this algorithm with parameter $q=150r$ (chosen to satisfy \eqref{eq:Ai}) to the bipartite subgraph $G_i[A_i, B_i] \subset G_i$ obtained from a uniformly random bipartition. Let $H_i$ be the $i$th colour class and set $G_{i+1}=G_i\setminus H_i$. 

To analyse this process, we define $e(G_i)=100\lambda_ir^2\log r k$, so that \[q=150r\leq20r/\lambda_i \qquad \text{and}\qquad \Delta(G_i[A_i,B_i])\leq\lambda_ir\log r\]
by the regularity assumption. Combining Observation~\ref{obs:H} and Fact~\ref{fact:lambdaballsunified}, it follows that
\[ e(H_i)\geq \frac{e(G_i)}{3}\cdot W_0(q,\Delta(G))\geq\frac{40e(G_i)}{\lambda_i^{0.9} r}.\]
This in turn implies that
\[\lambda_{i+1}r^2\log r k=e(G_{i+1})=e(G_i)-e(H_i)\leq \lambda_{i}r^2\log r k\left(1-\frac{40}{\lambda_i^{0.9} r}\right).\]
By using the bound $(1-x)^{0.9}\leq 1-0.9x$, we obtain the following recursion on $\lambda_i$:
\begin{equation}\label{eq:bootstrap}
    \lambda_{i+1}^{0.9}\leq \lambda_i^{0.9}-\frac{36}{r}.
\end{equation}
In words, since $\lambda_i$ tracks the average degree of this algorithm, \eqref{eq:bootstrap} shows that this process must terminate after at most $r/36$ steps, which corresponds to using at most $r/36$ many colours, before we have removed all the edges of $G$. That is, the boost of $\lambda_i^{0.9}$ exactly allowed us to do better than the naive approach in Section~\ref{sec:sketch-main} and obtain the desired $r$-colouring.

There are two white lies in the toy example above. As we have already explained, the first is the strong regularity assumption, and the second is the assumption that $v(G)=O(rk)$. The regularity assumption turns out to be the bigger issue. The first observation is that we may take $\Delta(G)$ in \eqref{eq:withconcentration} to be the maximum degree of $G$, and the same inequality stays true. By the monotonicity of $W_0(q,d)$, we would like to keep both $\Delta(G[A,B])$ and $|A|$ small to fit in the range of parameters in Fact~\ref{fact:lambdaballsunified}. We elaborate on how to do so in the next section, establishing Lemma~\ref{lem:main} that identifies a $P_k$-free subgraph with $\Omega(e(G)/(\beta^{0.9}\cdot r))$ edges, where $\beta=\Delta(G)/(r\log r)$.

\section{The key lemma}\label{sec:key-lemma} 

In this section, we state and prove the key lemma that allows us to extract dense $P_k$-free subgraphs.

\begin{restatable}[Key lemma]{lemma}{mainlemma}\label{lem:main}
There exists a constant $\beta_0 \in (0,1]$ such that for all sufficiently large $r$ and $k\geq 200\log r$ we have the following. Let $G$ be a graph with 
\[e(G) \leq (r^2 \log r) \, k\hspace{0.8em} \text{and} \hspace{0.8em}\Delta(G) = \beta r \log r\in \NN \]
where $\beta \in (0,\beta_0]$.
Then there exists a $P_k$-free subgraph $H \subset G$ such that \[e(H) \geq \dfrac{10 e(G)}{\beta^{0.9} r}.\] 
\end{restatable}

A crucial feature of Lemma~\ref{lem:main} is that a smaller $\beta$ gives rise to a larger fraction of edges in $H$ (i.e. a larger $e(H)/e(G)$), which is essential in deriving an analogue of \eqref{eq:bootstrap} in the case of a general graph $G$. For the application of Lemma~\ref{lem:main}, see Section~\ref{sec:main-pf}.

We prove Lemma~\ref{lem:main} in two steps. First, we show that \eqref{eq:Ai} and Observation~\ref{obs:H} can be guaranteed up to constants with positive probability, hence obtaining a statement on finding large $P_k$-free graphs in bipartite graphs. Then, we turn to proving Lemma~\ref{lem:main}, where the essential step is to identify the bipartite graph within which we apply the following lemma.

\begin{restatable}[Finding $H$ in a bipartite graph]{lemma}{biglem}\label{lem:bipartite-finding}
    For $r$ sufficiently large and $k\geq 200\log r$, let $G$ be a bipartite graph with $e(G) \leq (r^2 \log r)\,k$ that contains no isolated vertices. Let $V(G)=A\cup B$ be the vertex partition of $G$, and let $\Delta = \max_{v \in B} \deg(v)$. Then $G$ has a $P_k$-free subgraph $H$ such that
    \[e(H)\geq \frac{2e(G)}{3}\cdot W_0(q,\Delta) \qquad \text{where} \qquad q = \left \lfloor \frac{10|A|}{k}\right\rfloor.\]
\end{restatable}

\begin{proof}
For $i \in [q]$, we define $A_i$ and $B_i$ according to the process described in the previous section (see for instance \eqref{eq:process}). Then we define $H$ as in \eqref{eq:H-def}. We show that this $H$ satisfies the claimed properties with positive probability. 

Observe that for $i\in [q]$, we have
\[ \mathbb{E}\hspace{0.1em}|A_i|=\frac{|A|}{q}=|A| \cdot \left \lfloor \frac{10|A|}{k}\right\rfloor^{-1} \qquad \text{so} \qquad \frac{k}{10} <\mathbb{E} \hspace{0.1em}|A_i|\leq \frac{k}{5}.\]

Next, we define the event  $\cE = \{ |A_i| < k/2 - 1 : \forall \; i \in [q] \}$. Since $|A_i| \sim \Bin(|A|,1/q)$ for each $i\in [q]$, by the Chernoff bound we have
\[ \PP\big(|A_i| \geq k/2 - 1 \big)\leq \PP\big(|A_i| \geq 2\mathbb{E} \hspace{0.1em}|A_i|\big) \leq 
\exp (-\mathbb{E} \hspace{0.1em} |A_i|/3)\leq \exp(-k/30).\]
A simple fact from the definition of $W_0$ in \eqref{eq:W_0def} is $W_0(q,\Delta)\geq 1/(10q)$. Since $G$ contains no isolated vertices, $|A|\leq (r^2 \log r)\,k$ and $q\leq 10r^2 \log r\ll r^{5/2}$, so for sufficiently large $r$ we obtain
\[ \PP(\cE^c) \leq\sum_{i=1}^q\PP\big(|A_i|\geq k/2-1 \big)\leq  q \exp(-k/30 )\leq 10W_0(q,\Delta) \hspace{0.1em} q^2r^{-5}\leq W_0(q,\Delta)/3,\]
where we used $k\geq 200\log r$, 

Recall from Observation~\ref{obs:H} that  $\EE \hspace{0.1em} e(H)\geq e(A,B) \cdot W_0\left( q, \Delta(G)\right)$. By construction, we have $e(H)\leq e(A,B)$ and thus,
\[ \EE \hspace{0.1em}  e(H) \1(\cE)\geq e(A,B) \, W_0(q,\Delta) - e(A,B) \PP(\cE^c)\geq 2e(A,B) \, W_0(q,\Delta)/3. \]

This means that there exists a partition $A = A_1 \cup \cdots \cup A_q$ with $|A_i| < k/2 - 1$ for all $i \in [q]$ and an associated subgraph $H$ for which 
\[ e(H) \geq 2e(A,B) \, W_0(q,\Delta)/3 . \] 
In particular, by \eqref{eq:pk-free} it follows that $H$ is $P_k$-free as well. This completes the proof. 
\end{proof}

In the remainder of this section, we show how to deduce Lemma~\ref{lem:main} from Lemma~\ref{lem:bipartite-finding}. Given $G$ in the setting of Lemma~\ref{lem:main}, we aim to identify a bipartite subgraph of $G$ to apply Lemma~\ref{lem:bipartite-finding}. To select this subgraph we first find a partition 
\begin{equation}\label{eq:part-G}
    G = E_1\cup E_2\cup\dots \cup E_T,
\end{equation}
where each $E_j$ has a vertex partition $V(E_j)=V_j\cup U_j$ where $U_j$ is independent in $E_j$, and such that $E_j$  is ``approximately regular''. This gives us simultaneous control of $|V_j|$ and $\Delta(E_j)$, ensuring they lie in the range of parameters specified by Fact~\ref{fact:lambdaballsunified}.

We will select $j$ with $e(E_j)\geq e(G)/2^j$, then pass to a bipartite subgraph of $E_j$ with at least $e(E_j)/2$ edges. By applying Lemma~\ref{lem:bipartite-finding} to this bipartite subgraph, we obtain the desired large $P_k$-free subgraph with many edges.

\begin{proof}[Proof of Lemma~\ref{lem:main}] Let $\theta  >0$ be such that $\theta^{0.9}=1/2$, and let $T\in \mathbb{N}$ be minimal such that $\Delta(G)\cdot \theta^{T}< 1$. Consider the following process. In the first step, we define \[ V_1=\big\{v\in V(G): \deg_{G}(v)\in[\theta\cdot \Delta(G), \Delta(G)] \big\},\] and let $E_1\subset G$ be the subgraph containing all edges that are incident to a vertex of $V_1$. We note that $\Delta(G\setminus E_1)\leq \theta\cdot \Delta(G)$. Now, if for some $j\leq T$ we have already defined $V_1, \dots, V_{j-1}$ and $E_1, \dots, E_{j-1}$, we next define\[ V_j=\big\{v\in V(G)\setminus V_{\leq j-1}: \deg_{G\setminus E_{\leq j-1}}(v)\in [\theta^j\Delta(G),\theta^{j-1}\Delta(G)] \big\}.\] Then let $E_j\subset G\setminus E_{\leq j-1}$ be the subgraph containing all edges that are incident to a vertex of $V_j$. By induction, we can prove that $\Delta(G\setminus \bigcup_{i\leq j}E_i)\leq \Delta(G)\cdot \theta^{j}$. Crucially, $E(G)=\bigcup_{j\leq T}E_j$, as $\Delta(G\setminus \bigcup_{j\leq T}E_j)\leq \Delta(G)\cdot \theta^{T}<1$.
    
We now claim that $e(E_j)\geq e(G)/2^j$ for some $j\in [T]$. To see this, assume otherwise. Then
    \[e(G)=\sum_{j=1}^Te(E_j)<e(G)\sum_{j=1}^{\infty}2^{-j}=e(G),\] which is a contradiction.  

Let $j \in [T]$ be such that $e(E_j) \geq e(G)/2^j$. By construction we can partition $V(E_j)=V_j\cup U_j$, where $U_j$ is an independent set, and 
\[\Delta(E_j)\leq \theta^{j-1}\Delta(G)=\theta^{j-1}\beta r\log r\quad \text{and} \quad \deg_{E_j}(v) \geq \theta^j\Delta(G) \quad\text{for all $v\in V_j$.}\]
Now, choose a partition of $V_j = A_j \cup B_j$ where 
\[ |A_j| = \lceil |V_j|/2 \rceil \qquad \text{ and } \qquad e_{E_j}(A_j, B_j \cup U_j) \geq e(E_j)/2\] (by considering a uniformly random partition of $V_j$, for example). We now apply Lemma~\ref{lem:bipartite-finding} to the induced bipartite graph $E_j[A_j, B_j \cup U_j]$, where we take $q=\lfloor 10|A_j|/k\rfloor$, to obtain a $P_k$-free subgraph $H\subset E_j$ satisfying
    \[e(H)\geq \frac{e(E_j)}{6} \;W_0(q,\Delta(E_j)).\]

Let $\lambda_j=\theta^{j-1}\beta\leq \beta_0\leq \lambda_0$, so that $\Delta(E_j)\leq\lambda_j r\log r$. By double-counting the edges of $E_j$, we get \[|V_j|\leq \frac{2e(E_j)}{\theta^{j}\Delta(G)}.\]
Therefore
\[ q=\left\lfloor\frac{10|A_j|}{k}
\right\rfloor=\left\lfloor\frac{10}{k}\cdot\left\lceil\frac{|V_j|}{2}\right\rceil\right\rfloor \leq \left\lfloor\frac{10}{k}\cdot\frac{2e(E_j)}{\theta^{j}\Delta(G)}\right\rfloor\leq \left\lfloor\frac{20(r^2\log r)k}{\theta^j\beta (r\log r) k}\right\rfloor\leq\frac{20r}{\lambda_j}.\]

    By the monotonicity of $W_0(q,d)$ in $q$ (Fact~\ref{fact:W_0mono}), and $e(E_j)\geq e(G)/2^j=\theta^{0.9j}e(G)$, we have
     \[e(H)\geq \frac{e(E_j)}{6} \;W_0(q,\Delta) \geq \frac{\theta^{0.9j}e(G)}{6} \cdot W_0\left(\frac{20r}{\lambda_j},\lambda_jr\log r\right).\]
     Then by Fact~\ref{fact:lambdaballsunified} we have
   \[e(H) \geq \frac{\theta^{0.9j}e(G)}{6} \cdot \frac{120}{r\cdot (\theta^{j-1}\beta)^{0.9}}=\dfrac{\theta^{0.9}\cdot 20  e(G)}{\beta^{0.9} r}= \dfrac{10 e(G)}{\beta^{0.9} r}, \]
   as desired.
   \end{proof}

\section{Proof of the main theorem}\label{sec:main-pf}

In this section, we prove our main theorem, Theorem~\ref{thm:main}. That is, given a graph $G$ with
$e(G)\leq c_0(r^2\log r)k$, we show that there is an $r$-colouring of $G$ with no
monochromatic $P_k$.

As discussed above, our strategy is to iteratively extract $P_k$-free subgraphs $H_0,H_1,\dots,H_{t-1}$ using
Lemma~\ref{lem:main}. To ensure that the maximum degree decreases nicely as the process runs, in each round, we will move all edges incident with ``high-degree'' vertices into a ``remainder'' graph $R$. This will allow us to ensure all edges not set aside appear in one of the $H_i$. To colour the edges in $R$, we observe that the graph $R$ has a small vertex cover, namely the high degree vertices that we set aside. Because the vertex cover number of $P_k$ is rather large (of size $k/2$), this means we may colour the edges of $R$ by partitioning its vertex cover into parts of size at most $k/3$. We then colour all edges incident to a given part with the same colour. 

\begin{proof}[Proof of Theorem~\ref{thm:main}] It suffices to prove the theorem for sufficiently large $r$, since the finitely many smaller values of $r$ can be absorbed into the choice of the constant $c_0$. Let $\beta_0$ be given by Lemma~\ref{lem:main}, let $r$ be sufficiently large  and take $\alpha=8/9$. We let
\[ c_0 \leq  \frac{1-\alpha}{10}\beta_0 \qquad \text{and} \qquad k \geq 200\log r,\]
and let $G$ be a graph with $e(G)\leq c_0(r^2\log r)k$ edges. We show that there is an $r$-colouring of $G$ with no monochromatic $P_k$.

First, we iteratively remove vertices of $G$ that have degree higher than $\beta_0r\log r$, and move them into a set $V_{-1}$. Note that
\[|V_{-1}|\leq\frac{c_0r^2(\log r)k}{\beta_0r\log r}\leq \frac{(1-\alpha)rk}{10}\leq \frac{rk}{90}.\]
Let the resulting graph be $G_0$. Then $G_0$ satisfies
\[e(G_0)\leq c_0\, r^2(\log r)k\qquad\text{and}\qquad
\Delta(G_0)\leq \beta_0\, r(\log r).\]

For each $i\geq 0$, we will construct a graph $G_i$ with parameters
\begin{equation}\label{eq:prop-Gi} e(G_i)=c_i\, r^2(\log r)k
\qquad\text{and}\qquad
\Delta(G_i)\leq \beta_i\, r(\log r), \qquad\text{where}\qquad \beta_i=\beta_0\Bigl(\frac{c_i}{c_0}\Bigr)^\alpha,\end{equation}
as long as $\beta_i\geq1/(r\log r)$. Once $\beta_i<1/(r\log r)$, we terminate the process and set $t=i$. Note that $\Delta(G_t)<1$, so $G_t$ is the empty graph.

Given $G_i$ satisfying \eqref{eq:prop-Gi}, we may apply Lemma~\ref{lem:main} to $G_i$ with $\beta=\beta_i$, obtaining a $P_k$-free subgraph $H_i\subset G_i$ with
\[e(H_i)\geq \frac{10\,e(G_i)}{\beta_i^{0.9}r}.\]

Set $F_{i,0}:=G_i\setminus H_i$. For $j\geq 0$, while $F_{i,j}$ contains a vertex of degree more than
\[\beta_0\Bigl(\frac{e(F_{i,j})}{c_0 r^2(\log r)k}\Bigr)^\alpha r(\log r),\]
we delete one such vertex and call the resulting graph $F_{i,j+1}$.
When this process stops, let $V_i$ be the set of deleted vertices, let $G_{i+1}=F_{i,|V_i|}$ and write $e(G_{i+1})=c_{i+1}\,r^2(\log r)k$. We
 then have by construction
\[\Delta(G_{i+1})\leq \beta_0\Bigl(\frac{c_{i+1}}{c_0}\Bigr)^\alpha r(\log r)
=\beta_{i+1}r(\log r).\]

In order to bound the number of iterations, we will need the following recursion on $\beta_i$.

\begin{claim}\label{claim:ci}
    We claim that 
    \[\beta_i^{0.9}\leq \beta_0^{0.9}-\frac{8 i}{r}.\]
\end{claim}

\begin{proof}
    Since $e(G_{i+1})\leq e(F_{i,0})=e(G_i)-e(H_i)$, we have
\[c_{i+1}\leq c_i\left(1-\frac{10}{\beta_i^{0.9}r}\right).\]
As long as the process has not terminated, we have $\beta_i\geq 1/(r\log r)$, and hence $\frac{10}{\beta_i^{0.9}r} <1$ for all sufficiently large $r$. Using $(1-x)^{0.8}\leq 1-0.8x$ for $x\in[0,1]$, we obtain
\[\beta_{i+1}^{0.9}=
\beta_0^{0.9}\Bigl(\frac{c_{i+1}}{c_0}\Bigr)^{0.8} \leq
\beta_0^{0.9}\Bigl(\frac{c_i}{c_0}\Bigr)^{0.8}
\left(1-\frac{10}{\beta_i^{0.9}r}\right)^{0.8} =
\beta_i^{0.9}\left(1-\frac{10}{\beta_i^{0.9}r}\right)^{0.8}\leq
\beta_i^{0.9}-\frac{8}{r}.\]
Iterating, we get
\[\beta_i^{0.9}\leq \beta_0^{0.9}-\frac{8 i}{r},\]
as desired.
\end{proof}

\begin{corollary}
    The iteration stops with $t \leq r/8$.
\end{corollary}
\begin{proof}
    As $\beta_t$ is a positive real number, and $\beta_0<1$, we get $t\leq r\cdot\beta_0^{0.9}/8\leq r/8$.
\end{proof}

Next, we control the total number of deleted vertices.

\begin{claim}
   We have
    \[\sum_{i=0}^{t-1}|V_i|\leq \frac{rk}{10}.\]
\end{claim}

\begin{proof} 
 For $0\leq j\leq |V_i|$, let $e(F_{i,j})=d_{i,j}\,r^2(\log r)k$, so when the vertex deleted in passing from $F_{i,j}$ to $F_{i,j+1}$ is removed, its degree is more than $\beta_0\Bigl(\frac{d_{i,j}}{c_0}\Bigr)^\alpha r(\log r)$. Hence
\[(d_{i,j}-d_{i,j+1})\,r^2(\log r)k
\geq \beta_0\Bigl(\frac{d_{i,j}}{c_0}\Bigr)^\alpha r(\log r),\]
and therefore
\[\frac{\beta_0}{rkc_0^{\alpha}}
\leq \frac{d_{i,j}-d_{i,j+1}}{d_{i,j}^\alpha}\leq \int_{d_{i,j+1}}^{d_{i,j}}x^{-\alpha}\,dx
=\frac{d_{i,j}^{1-\alpha}-d_{i,j+1}^{1-\alpha}}{1-\alpha},\]
where we used that $x\mapsto x^{-\alpha}$ is decreasing. Note that $d_{i,0}
\leq c_i$ and $d_{i,|V_i|}=c_{i+1}$.
Summing over $0\leq j\leq |V_i|$, we obtain
\[\frac{|V_i|}{rk}
\leq \frac{c_0^\alpha}{\beta_0(1-\alpha)}
\sum_{j=0}^{|V_i|-1}\bigl(d_{i,j}^{1-\alpha}-d_{i,j+1}^{1-\alpha}\bigr)
\leq \frac{c_0^\alpha}{\beta_0(1-\alpha)}
\bigl(c_i^{1-\alpha}-c_{i+1}^{1-\alpha}\bigr).\]
By summing over $0\leq i<t$, we conclude that
\[\sum_{i=0}^{t-1}\frac{|V_i|}{rk}
\leq
\frac{c_0^\alpha}{\beta_0(1-\alpha)}
\sum_{i=0}^{t-1}\bigl(c_i^{1-\alpha}-c_{i+1}^{1-\alpha}\bigr)
\leq
\frac{c_0}{\beta_0(1-\alpha)}
\leq \frac1{10},\]
establishing the second part of the claim.    
\end{proof}

Now, we are ready to describe the colouring of $G$. Note that the edges of $G$ can be partitioned as
\[E(G)=E(R)\cup\bigcup_{i=0}^{t-1}E(H_i),\]
where $R$ is a subgraph of $G$ which has a vertex cover given by $C=\bigcup_{i=-1}^{t-1}V_i$.

We first colour $\bigcup_{i=0}^{t-1}E(H_i)$. Since $t\leq r/8$ and each $H_i$ is $P_k$-free, assigning a distinct colour to each $H_i$ gives a monochromatic $P_k$-free colouring using at most $r/8$ colours.

It remains to colour $E(R)$. Note that \[|C|=\sum_{i=-1}^{t-1}|V_i|\leq \frac{rk}{90}+\frac{rk}{10}=\frac{rk}{9},\] we can partition $C$ into at most $r/4$ sets, each of size less than $k/2-1$. By assigning each edge of $R$ to a colour corresponding to the index of either one of its endpoints, we obtain a colouring of $E(R)$ with $r/4$ colours, where we break ties arbitrarily. Each colour class is a vertex-disjoint union of at most $k/2-1$ stars, so this gives a monochromatic $P_k$-free colouring.

Combining the two colourings above gives a monochromatic $P_k$-free colouring of $E(G)$
using at most $r/4+r/8<r$ colours
\end{proof}

\section*{Acknowledgements}
We would like to thank Marcelo Campos for insightful discussions. AL and CB are grateful to Dylan Toh for his helpful feedback on an earlier draft of this paper. AL would also like to thank the Mathematics Department at the University of Cambridge for their hospitality and the Stanford EDGE Doctoral Fellowship Program for their generous support of her research visits to Cambridge.

\bibliographystyle{plain}
\bibliography{ref.bib}

@article{K19,
    AUTHOR = {Krivelevich, Michael},
     TITLE = {Long cycles in locally expanding graphs, with applications},
   JOURNAL = {Combinatorica},
  FJOURNAL = {Combinatorica. An International Journal on Combinatorics and
              the Theory of Computing},
    VOLUME = {39},
      YEAR = {2019},
    NUMBER = {1},
     PAGES = {135--151},
      ISSN = {0209-9683,1439-6912},
   MRCLASS = {05C38 (05C35 05C57 05C80 05D10)},
  MRNUMBER = {3936195},
MRREVIEWER = {John\ Haslegrave},
       DOI = {10.1007/s00493-017-3701-1},
       URL = {https://doi.org/10.1007/s00493-017-3701-1},
}

@article{EFRS78,
    AUTHOR = {Erd{\H o}s, P. and Faudree, R. J. and Rousseau, C. C. and
              Schelp, R. H.},
     TITLE = {The size {R}amsey number},
   JOURNAL = {Period. Math. Hungar.},
  FJOURNAL = {Periodica Mathematica Hungarica. Journal of the J\'anos Bolyai
              Mathematical Society},
    VOLUME = {9},
      YEAR = {1978},
    NUMBER = {1-2},
     PAGES = {145--161},
      ISSN = {0031-5303,1588-2829},
   MRCLASS = {05C55},
  MRNUMBER = {479691},
MRREVIEWER = {F.\ Harary},
       DOI = {10.1007/BF02018930},
       URL = {https://doi.org/10.1007/BF02018930},
}

@article{B83,
    AUTHOR = {Beck, J\'ozsef},
     TITLE = {On size {R}amsey number of paths, trees, and circuits. {I}},
   JOURNAL = {J. Graph Theory},
  FJOURNAL = {Journal of Graph Theory},
    VOLUME = {7},
      YEAR = {1983},
    NUMBER = {1},
     PAGES = {115--129},
      ISSN = {0364-9024,1097-0118},
   MRCLASS = {05C55},
  MRNUMBER = {693028},
MRREVIEWER = {Saul\ Stahl},
       DOI = {10.1002/jgt.3190070115},
       URL = {https://doi.org/10.1002/jgt.3190070115},
}

@incollection{CFS15,
    AUTHOR = {Conlon, David and Fox, Jacob and Sudakov, Benny},
     TITLE = {Recent developments in graph {R}amsey theory},
 BOOKTITLE = {Surveys in combinatorics 2015},
    SERIES = {London Math. Soc. Lecture Note Ser.},
    VOLUME = {424},
     PAGES = {49--118},
 PUBLISHER = {Cambridge Univ. Press, Cambridge},
      YEAR = {2015},
      ISBN = {978-1-107-46250-2},
   MRCLASS = {05-02 (05C55)},
  MRNUMBER = {3497267},
}

@article{DP17,
    AUTHOR = {Dudek, Andrzej and Pra{\l}at, Pawe{\l}},
     TITLE = {On some multicolor {R}amsey properties of random graphs},
   JOURNAL = {SIAM J. Discrete Math.},
  FJOURNAL = {SIAM Journal on Discrete Mathematics},
    VOLUME = {31},
      YEAR = {2017},
    NUMBER = {3},
     PAGES = {2079--2092},
      ISSN = {0895-4801,1095-7146},
   MRCLASS = {05D10 (05C55 05C80)},
  MRNUMBER = {3697158},
       DOI = {10.1137/16M1069717},
       URL = {https://doi.org/10.1137/16M1069717},
}

@incollection{RS98,
    AUTHOR = {Raab, Martin and Steger, Angelika},
     TITLE = {``{B}alls into bins''---a simple and tight analysis},
 BOOKTITLE = {Randomization and approximation techniques in computer science
              ({B}arcelona, 1998)},
    SERIES = {Lecture Notes in Comput. Sci.},
    VOLUME = {1518},
     PAGES = {159--170},
 PUBLISHER = {Springer, Berlin},
      YEAR = {1998},
      ISBN = {3-540-65142-X},
   MRCLASS = {68R05 (68M20)},
  MRNUMBER = {1729169},
       DOI = {10.1007/3-540-49543-6\_13},
       URL = {https://doi.org/10.1007/3-540-49543-6_13},
}

@book{MU17,
    AUTHOR = {Mitzenmacher, Michael and Upfal, Eli},
     TITLE = {Probability and computing},
   EDITION = {Second},
      NOTE = {Randomization and probabilistic techniques in algorithms and
              data analysis},
 PUBLISHER = {Cambridge University Press, Cambridge},
      YEAR = {2017},
     PAGES = {xx+467},
      ISBN = {978-1-107-15488-9},
   MRCLASS = {68-01 (60C05 60G42 60J10 60K25 62H30 68W20 68W40)},
  MRNUMBER = {3674428},
}

@article{G81,
    AUTHOR = {Gonnet, Gaston H.},
     TITLE = {Expected length of the longest probe sequence in hash code
              searching},
   JOURNAL = {J. Assoc. Comput. Mach.},
  FJOURNAL = {Journal of the Association for Computing Machinery},
    VOLUME = {28},
      YEAR = {1981},
    NUMBER = {2},
     PAGES = {289--304},
      ISSN = {0004-5411,1557-735X},
   MRCLASS = {68E05 (68H05)},
  MRNUMBER = {612082},
       DOI = {10.1145/322248.322254},
       URL = {https://doi.org/10.1145/322248.322254},
}

@techreport{Steel53,
  author       = {Steel, Robert G. D.},
  title        = {Relation Between Poisson and Multinomial Distributions},
  institution  = {Cornell University, Biometrics Unit},
  number       = {BU-39-M},
  year         = {1953},
  month        = apr,
}

\section{Deferred balls-and-bins analysis}\label{app:balls-bins}

In this appendix, we study the \emph{balls and bins problem}, and we refer the reader to the main body of this paper for a short list of references for this problem. 

We begin by briefly recalling the set-up of the problem and our notation for the various random variables. Suppose that we sequentially throw $d$ balls into $q$ bins by placing each ball into a bin chosen independently and uniformly at random. Here we study the maximum number of balls in a bin, which we denote by $M_{q,d}$. In our application, we are mostly interested in $W(q,d) = \mathbb{E}\, M_{q,d}/d$. We first prove an efficient bound for $\mathbb{E}\, M_{q,d}$ in Section~\ref{sec:appMbound}, and then establish some monotonicity properties of $W(q,d)$ in Section~\ref{sec:appWprop}.

\subsection{Bounds on $\mathbb{E}\, M_{q,d}$}\label{sec:appMbound}

\begin{fact}\label{fact:W-ez}
    For any $q,d\in \NN$, we have the bound $\mathbb{E} M_{q,d} \geq \max\big\{ \tfrac{d}{q}, 1 \big\}$.
\end{fact}
\begin{proof}
    By the pigeonhole principle, no matter how the balls are distributed among the bins, there will be one with at least $d/q$ balls in it. It is also clear that there will be at least one non-empty bin. Therefore, $M_{q,d} \geq \max\big\{ \tfrac{d}{q}, 1 \big\}$ and $\mathbb{E} M_{q,d} \geq \max\big\{ \tfrac{d}{q}, 1 \big\}$.
\end{proof}

Next, we prove the driving force behind the key lemma in our paper; we restate Lemma~\ref{lem:ballsunified} in an equivalent form for the reader's convenience. 

\begin{lemma}[Lemma~\ref{lem:ballsunified}]\label{lem:ballsunifed-appendix}
   For $d,q \in \NN$ with $q > 1$, let $x>1$ be the unique solution to $\frac{q \log q}{d}=x \log x$. Then $\mathbb{E}\,M_{q,d} \geq \frac{xd}{10q}$.
\end{lemma}

We use the following ``Poissonization trick'', which relies on a simple observation. 
A proof can be found, for example, in Steel~\cite{Steel53}.

\begin{fact}
Let $P_1,\dots,P_q$ be independent Poisson random variables with means $\lambda_1,\dots,\lambda_q$. Set $\Lambda=\sum_i \lambda_i$. Then, conditional on $\sum_i P_i=d$, the vector $(P_1,\dots,P_q)$ is distributed as the multinomial random variable $\Mult\!\left(d;\lambda_1/\Lambda,\dots,\lambda_q/\Lambda\right)$. 
\end{fact}

Therefore, if we take $P_1,\dots,P_q$ be to be independent Poisson random variables with the same mean, then conditional on $\sum_{i=1}^q P_i  = n$, $(P_1,\dots,P_q)$ has the same distribution as in the balls-and-bins experiment with $n$ balls and $q$ bins.

\begin{proof}[Proof of Lemma~\ref{lem:ballsunifed-appendix}]
In the regimes $q\log q\leq d$, $d\leq 10$ or $xd/q\leq 2$ the claim follows immediately by Fact~\ref{fact:W-ez}. Therefore, we only focus on the case when $q\log q\geq d$, $d\geq 11$ and $xd/q\geq 2$.






Now let $\mu=d/(2q)$, and let $P_1,\dots,P_q$ be independent $\Pois(\mu)$ random variables. Write 
\[ N =\sum_{i=1}^q P_i \qquad \text{ and } \qquad Z=\max_{i\in[q]} P_i.\] Note that $N\sim \Pois(d/2)$.

For any $d\geq n$ and $t>0$, we have $\PP(M_{q,d}\geq t)\geq \PP(M_{q,n}\geq t)$, as throwing $d-n$ extra balls independently and uniformly at random can only increase the load of the maximum bin. Therefore,
\[\PP(N\leq d,Z\geq t) =\sum_{n=0}^d\PP(N=n,Z\geq t)=\sum_{n=0}^d\PP(N=n)\cdot \PP(M_{q,n}\geq t)\leq \PP(M_{q,d}\geq t).\]
Furthermore,
\[\mathbb E M_{q,d}\geq t\,\PP(M_{q,d}\geq t)\geq t\cdot \PP(N\leq d,Z\geq t)\geq t\cdot \left(\PP(Z\geq t)-\PP(N>d)\right).\]

 Let $p=\PP(\Pois(\mu)\geq t)\geq \PP(\Pois(\mu)=t)$. By independence of the variables $P_i$, we have
\[\PP(Z\geq t)=1-(1-p)^q\geq 1-e^{-qp} \quad \text{and} \quad p\geq e^{-\mu}\frac{\mu^t}{t!}\geq e^{-\mu}\left(\frac{\mu}{t}\right)^t.\]
Now let $t=\lfloor xd/(2q)\rfloor$. Then $t\leq xd/(2q)=\mu x$, so $(\mu/t)^t\geq x^{-t}\geq x^{-xd/(2q)}=q^{-1/2}$. Also, as $q\log q\geq d$, we have $e^{-\mu}=e^{-d/(2q)}\geq e^{-\log(q)/2}=q^{-1/2}$. Therefore $p\geq q^{-1}$ and
\[\PP(Z\geq t)\geq 1-e^{-qp}\geq 1-e^{-1}.\]

Since $N\sim \Pois(d/2)$, it follows that $\mathbb{E}(N) = d/2$ and $\mathrm{Var}(N) = d/2$. By Chebyshev's inequality, it follows that 
\[ \mathbb{P}(N>d) \leq \mathbb{P} \biggl( \bigg| N - \frac{d}{2} \biggr| \geq \frac{d}{2} \biggr) \leq \frac{\mathrm{Var}(N)}{(d/2)^2}= \frac{2}{d} \leq \frac{2}{11},\]
where the final inequality follows from $d \geq 11$. 

Since $xd/q\geq 2$, we have $t=\lfloor xd/(2q)\rfloor\geq xd/(4q)$ and
\[\mathbb E M_{q,d}\geq t\cdot \left(\PP(Z\geq t)-\PP(N>d)\right)\geq \frac{xd}{4q}\cdot \left(1-e^{-1}-\frac{2}{11}\right)\geq \frac{xd}{10q}.\qedhere\]
\end{proof}

\subsection{Properties of $W_0(q,d)$}\label{sec:appWprop}

In this section, we prove the properties of the function $W_0(q,d)$ that we omitted before. To recall, for $q,d\in\NN$ we define

\[W_0(q, d)=\frac{x}{10q} \qquad \text{where} \qquad x\log x= \frac{q \log q}{d}.\]

\mono*

\begin{proof}
It is easy to see that $W_0(q,d)$ is decreasing in $d$. Simply note that since 
 $x\log x = (q\log q)/d$, as $d$ increases $x$ decreases and thus $W_0(q,d) = x/(10q)$ decreases.

To prove that $W_0$ is decreasing in $q$, we put $f(q) = x(q)/q$. It is enough to show $f(q)$ is decreasing in $q$. Using that $x(q)\log x(q) = (q\log q)/d$ we may write 
\[ (f(q)q)\log f(q)q = (q \log q)/d. \]
Also note that since $d \geq1 $ and $x\log x  = (q\log q)/d$, we have $x(q) \leq q$. Thus $0 < f(q) \leq 1$ and $\log f(q) \leq 0$.

Now, for a contradiction, suppose there exists $q' > q > 1$ such that
$f(q') \geq f(q)$. Rearranging the above gives
\[  f(q)\Big( 1 + \frac{\log f(q)}{\log q} \Big) = f(q')\Big( 1 + \frac{\log f(q')}{\log q'} \Big)  =  \frac{1}{d} \]
We note that 
\[ \frac{\log f(q')}{\log q'} > \frac{\log f(q)}{\log q'} \geq \frac{\log f(q)}{\log q } .\] The first inequality follows from the assumption $\log f(q') \geq \log f(q)$. The second inequality follows from the fact $\log f(q') \leq  0 $ and the assumption $q'>q$ and thus $-(\log q')^{-1} >- (\log q)^{-1}$.

Using this in the above gives $f(q) > f(q')$, which is a contradiction. 
\end{proof}

\lambdaballsandbins*

\begin{proof}
    Let $c\in (1/2,1]$ be, so that $q_0=\left\lfloor20r/\lambda\right\rfloor=c\cdot 20r/\lambda$.
    Let $x_0$ be the solution to $\frac{q_0 \log q_0}{d}=x_0 \log x_0$. By the monotonicity property of $W_0(q,d)$, we have
    \[W_0(q,d)\geq W_0(q_0,d)\geq \frac{x_0}{10q_0}\]
    so it suffices to show that $x_0\geq x'$, where $x'=\frac{1200q_0}{\lambda^{0.9}r}$. As $f(x)=x\log x$ is increasing for $x>1$, we seek $x'\log x'\leq x_0\log x_0=\frac{q_0\log q_0}{d}$. This can be seen to be equivalent to
    \[1200 \lambda^{0.1}\log\left(\frac{c\cdot 24000}{\lambda^{1.9}}\right)\leq 1+\frac{\log(20c/\lambda)}{\log r}\]
    The right-hand side is at least 1, while the left-hand side goes to $0$ as $\lambda \to 0$. Therefore, for a suitable choice of $\lambda_0$, the statement follows. 
\end{proof}

\end{document}